\documentclass{elsarticle}

\usepackage{lineno,hyperref}
\modulolinenumbers[5]

\journal{Applied Mathematics and Computation}

\usepackage{amsfonts}
\usepackage{amssymb}
\usepackage{amsthm}
\usepackage{amsmath}

\usepackage{graphicx}
\usepackage{float}

\def\Xo{X^0} \def\Xi{X^{\infty}}
\def\Wo{W^0} \def\Wi{W^{\infty}}
\def\No{N^0} \def\Ni{N^{\infty}}
\def\Fo{F^0} \def\Fi{F^{\infty}}
\def\Co{C_k^0} \def\Ci{C_k^{\infty}}
\def\muo{\mu^{0}} \def\mui{\mu^{\infty}}
\def\mo{m^{0}} \def\mi{m^{\infty}}
\def\bb{\beta_0}
\def\sbb{\frac{1}{2}\sigma^2\beta_0^2}

\def\eqn#1{\begin{equation}#1\end{equation}}

\def\Ind#1{I(#1)}
\def\th{\theta}
\def\s{\sigma}
\def\q{\qquad}
\def\d{\mathrm{d}}
\def\gsr{\widetilde{\phi}}

\def\P{{\mathbb P}}   
\def\E{{\mathbb E}}   \def\R{{\mathbb R}}
\def\N{{\mathbb N}}

\def\FF{\mathcal{F}}

\def\Px{\P^x} \def\Ex{\E^x} \def\Pxs{\Px_s} 
\def\Pxo{\Px_0} \def\Exo{\Ex_0} \def\Pxi{\Px_{\infty}} \def\Exi{\Ex_{\infty}}
\def\Pxt{\Px_t} 

\newcommand{\exit}{{\mbox{\, \vspace{3mm}}}\hfill\mbox{$\square$}}

\newtheorem{Thm}{Theorem}

\newtheorem{Lemma}{Lemma}

 \newtheorem{Problem}{Problem}

\begin{document}

\begin{frontmatter}

\title{Quickest drift change detection in L\'evy-type force of mortality model}
\tnotetext[grants]{This work is partially supported by National Science Centre under grants No. 2016/23/B/HS4/00566  (2017-2020) and No. 2016/23/N/HS4/02106 (2017-2019).}

\author[imuwr]{Micha\l \ Krawiec}
\ead{michal.krzysztof.krawiec@gmail.com}

\author[pwr]{Zbigniew Palmowski}
\ead{zbigniew.palmowski@gmail.com}

\author[pwr]{\L ukasz P\l ociniczak}
\ead{lukasz.plociniczak@pwr.edu.pl}

\address[imuwr]{Mathematical Insititute, University of Wroc\l aw, pl. Grunwaldzki 2/4, 50-384 Wroc\l aw, Poland}
\address[pwr]{Faculty of Pure and Applied Mathematics, Wroc\l aw University of Science and Technology, ul. Wyb. Wyspia\'nskiego 27, 50-370 Wroc\l aw, Poland}

\begin{abstract}
In this paper we give solution to the quickest drift change detection problem
for a L\'evy process consisting of both a continuous Gaussian part and a jump component.
We consider here Bayesian framework
with an exponential a priori distribution of the change point
using an optimality criterion based on a probability of false alarm and an expected delay of the detection. Our approach is based on the optimal stopping theory
and solving some boundary value problem.
Paper is supplemented by an extensive numerical analysis related with the construction of the Generalized Shiryaev-Roberts statistics.
In particular, we apply this method (after appropriate calibration) to analyse Polish life tables and to model the force of mortality in this population with a drift changing in time.
\end{abstract}

\end{frontmatter}

\newpage
\section{Introduction}
Quickest detection problems, often called also \textit{disorder problems}, arise in various fields of applications of mathematics, such as finance, engineering or economy.
All of them concern statistical methods of detection that allow to find some changes of observed system as quickly as possible.
One of the first motivations to study such problems was to distinguish incoming signal from noise. The main method was based
on the drift change detection using Bayesian approach. First model of this kind in continuous time was presented by
Shiryaev \cite{shiryaev1961problem, shiryaev1963optimum}, where Brownian motion with linear drift was considered and the drift changes according to an exponential distribution.
It was reformulated in terms of free boundary problem and solved using optimal stopping methods. Wide description of this problem with the historical background was also given by Shiryaev many years later in \cite{shiryaev2006disorder, shiryaev2010quickest} (see also references therein).
Later, the minimax approach have also been used.
This method is based on the identifying the optimal detection time based on so-called cumulative sums (CUSUM) strategy; see e.g. Beibel \cite{beibel1996note}, Shiryaev \cite{0036-0279-51-4-L15} or Moustakides \cite{moustakides2004optimality} in the Wiener case, or El Karoui et al. \cite{el2015minimax} in the Poisson case.
Many of these quickest detection problems and used methods in various settings were gathered in the book of Poor and Hadjiliadis \cite{poor2009quickest}.

This article has {\it three main goals}. {\it The first} one concerns solving {\it the quickest drift change detection problem for a L\'evy model under the Bayesian set-up}.
A good deal of work on the detection problems has been devoted to the Brownian or diffusion processes, i.e. to the processes with continuous trajectories; see
e.g. Beibel \cite{beibel1994bayes} or Shiryaev \cite{shiryaev1963optimum}, \cite[Chap. 4]{shiryaev2007optimal} and references therein.
In this paper we consider more general L\'evy process instead.
We assume that unobservable moment $\th$ of the drift change follows some exponential distribution with the parameter $\lambda >0$ (conditioned that it is strictly positive)
and it has additional atom at zero with a mass $x>0$. We consider the process
\eqn{\label{eq:X_0_inf}X_t=\left\{\begin{array}{ll} \Xi_t, & t<\th, \\ \Xi_{\th} + \Xo_{t-\th}, & t \geq \th, \end{array} \right.}
where $\Xi_t$ and $\Xo_t$ are both independent L\'evy processes.
We choose the optimality criterion based on both false alarm probability and mean delay time, that is, in this paper, we will find
the optimal detection rule $\tau^*\in \mathcal{T}$ for which the following infimum is attained:
\eqn{V^*(x)=\inf_{\tau\in \mathcal{T}}\left\{\Px(\tau<\th) + c\Ex[(\tau-\th)^+]\nonumber\right\},}
where $\mathcal{T}$ is the family of stopping times with respect to the
natural filtration $\FF_t^X$ of $X$ satisfying the usual conditions.
In recent years, the study of these processes has enjoyed rejuvenation.
Particular cases (apart from the Brownian case mentioned above)
have been already analysed e.g. by Gal’chuk and Rozovskii \cite{gal1971disorder}, Peskir and Shiryaev \cite{peskir2002solving} or Bayraktar et al. \cite{bayraktar2005standard} for the Poisson process, by Gapeev \cite{gapeev2005disorder} for
the compound Poisson process with the exponential jumps or by Dayanik and Sezer \cite{dayanik2006compound}
for more general compound Poisson problem.
This paper is {\it the first paper} dealing (in Bayesian framework) with {\it nontrivial both components}: Gaussian one and the jump one.
The {\it used method} is based on transferring the detection problem into
the following optimal stopping problem
\eqn{V^*(x)=\inf_{\tau\in\mathcal{T}}\Ex\left[1-\pi_{\tau} + c\int_0^{\tau}\pi_sds\right],\nonumber}
for the \textit{a posteriori} probability
$\pi_t=\Px(\th \leq t|\FF_t^X)$, where the superscript $x$ associated with $\P$ and $\E$ indicates the starting position $\pi_0=x$.
In the next step, using the change of measure technique, we can identify the infinitesimal generator of the Markov process $\pi_t$.
Finally, we formulate the boundary value problem and leaning on Frobenius theory
we solve it for the case of nontrivial Gaussian volatility and the exponentially distributed jumps.
We prove that the optimal detection rule is of threshold type for the process $\pi_t$
for some level $A^*$ that we can identify numerically.

Our {\it second main goal} is to apply the solution of
above detection problem to the analysis of {\it the change of drift in force of mortality}.
We proceed as follows. We take logarithm of force of mortality and subtract observed drift that can be calibrated from the historical data.
As in the seminal Lee-Carter model \cite{lee1992modeling}, we assume that this log-mortality model is perturbed by some noise.
In the Lee-Carter model there is a Gaussian perturbation.
In our model this perturbation comes from  0-mean L\'evy process. To this process we apply the above described quickest detection procedure to detect in optimal way the change of drift.
Precisely, we construct a statistical and numerical procedure based on the generalized version of the Shiryaev-Roberts statistics introduced by Shiryaev \cite{shiryaev1961problem, shiryaev1963optimum} and Roberts \cite{roberts1966comparison}, see also
Polunchenko and Tartakovsky \cite{polunchenko2012state}, Shiryaev \cite{shiryaev2002quickest}, Pollak and Tartakovsky \cite{pollak2009optimality} and Moustakides et al. \cite{moustakides2009numerical}.
Our generalization is based on the fact that we do not pass to the limit with the parameter of \emph{a priori} distribution of (exponentially distributed) random drift change moment.
The construction of this statistics is also new. We start from a continuous statistics derived from the solution of the optimal detection problem in continuous time.
Let $\phi_t=\frac{\pi_t}{1-\pi_t}$.
Then we take discrete moments of time $0<t_1<t_2<\ldots<t_N$ and we construct Generalized Shiryaev-Roberts (GSR) statistics by the following recursion:
\eqn{\phi_{t_{n+1}}\approx\gsr_{n+1}:=e^{\lambda+\beta_0x_{n+1}-\psi_{\infty}(\beta_0)}(\gsr_n+\lambda),\quad n=1,2,\ldots\nonumber}
for $x_{n+1}=X_{t_{n+1}}-X_{t_n}$ and $\beta_0$, $\psi_{\infty}(\beta_0)$ given explicitly.
Since the optimal stopping time in our detection problem is the first time when \textit{a posteriori} probability $\pi_t$ exceeds certain threshold $A^*$,
then it is also optimal to stop when GSR statistics $\gsr_n$ exceeds threshold $B^*:=A^*/(1-A^*)$.

{\it The third goal of our paper} is to provide an extensive numerical analysis of the (Polish) life tables.
Of course one can choose any other set of life tables to perform this analysis.
We start from a historical
calibration of our model assuming nontrivial Gaussian component. We also assume that jumps follows double-sided exponential distribution, allowing very quick increase or decrease in the force of mortality. Then we find the optimal threshold $B^*$ and apply
GSR statistics to detect changes in mortality. We show on many figures that presented
algorithm is very efficient. The problem of analysing the drift change in the force of mortality is important for the national and world economy because of observed longevity.
Nowadays, the insurance industry is facing huge challenges related to the improvements of longevity, which has significantly changed during the last decade. More and more capital has to be constituted to face this long-term risk and
new ways to cross-hedge or to transfer part of the
longevity risk to reinsurers or to financial markets need to be created.
Longevity risk is, however, not easy to transfer. To perform accurate longevity projections one has to identify the change of the drift observed in prospective life tables (national or the specific ones used in insurance companies).
To show how important the problem is, one can look at the French prospective life tables that were updated in 2006 after previous update done in 1993.
After this update, French insurers increased their reserves by 8 percentage on average to account
for the longevity phenomenon. Of course one can expect that there will be periods in which mortality is rapidly decreasing (e.g. because of medical discoveries or political changes) as well as periods in which it stays at a stable level. Still, we have to detect the moment in time when this change is really statistically substantial.
We believe that our optimal detection procedure comes in hand here.

The paper is organized as follows. In Section \ref{sec:setting} we describe basic setting of the problem and introduce main definitions and notation.
In this section we also formulate main theoretical results of the paper, which proofs are given in the Appendix (Section \ref{sec:app}). Section \ref{sec:Cal_USR} is devoted to the construction of the \textit{Generalized Shiryaev-Roberts} statistics. Next, in Section \ref{sec:numerical}, we provide extensive numerical analysis based on a real life tables data.
In particular, we explain there how to calibrate our mortality model. We finish our paper with conclusions given in Section \ref{sec:con}.

\section{Model and main result}\label{sec:setting}
Let $(\Omega,\FF,\Px)$ be a probability space on which we define random variable $\th$ being unobservable moment of drift change in our quickest detection model
and a process $X_t$, independent of $\th$, being the main process of our interest.
Both of these quantities we specify below.
We assume that this random drift change time $\th$ has an atom at $0$ with the probability $x$, that is,
\eqn{\label{eq:defpi}\Px(\th=0)=x\in[0,1]}
and we assume that, conditioned that $\th$ is positive, it has the exponential distribution with parameter $\lambda >0$:
\eqn{\Px(\th>t|\th>0)=e^{-\lambda t}.\nonumber}
On $(\Omega,\FF,\Px)$ we also introduce the process $X=(X_t)_{t\geq 0}$ of a random perturbation as follows:
\eqn{\label{eq:X}X_t=\s W_t + r(t-\th)^+ + \sum_{k=1}^{N_t}C_k(t) - \nu_t,}
where
\begin{itemize}
\item
\eqn{\label{sigma}\s>0}
and $W_t$ is a standard Brownian motion;
\item $r\in\R\backslash\{0\}$ is an additional linear drift rate, which comes after random time $\th$;
\item $N_t$ is a counting process consisting of two Poisson processes switching at time $\th$, i.e.
\eqn{N_t=\left\{\begin{array}{ll} \Ni_t, & t<\th, \\ \Ni_{\th} + \No_{t-\th}, & t \geq \th, \end{array} \right.\nonumber}
where $\Ni_t$ and $\No_t$ are independent Poisson processes with intensities $\mui$ and $\muo$, respectively;
\item
\eqn{C_k(t)=\left\{\begin{array}{ll}\Ci, & t<\th, \\ \Co, & t\geq\th, \end{array}\right.\nonumber}
where
$\{\Ci\}_{k\geq 1}$ is a sequence of i.i.d. random variables with distribution $\Fi(y)$ such that $\Ex[\Ci]=\int_{\R}y\d\Fi(y)=\mi<\infty$. Similarly, $\{\Co\}_{k\geq 1}$ is an independent of $\{\Ci\}_{k\geq 1}$ sequence
of i.i.d. random variables with distribution $\Fo(y)$ and mean $\Ex[\Co]=\int_{\R}y\d\Fo(y)=\mo<\infty$;
\item $\nu_t$ is a compensator of the jump process:
$$\nu_t=\mui \mi\int_0^t\Ind{\th\geq s}\d s+\muo \mo\int_0^t\Ind{\th<s}\d s.$$
\end{itemize}

In other words, process $X_t$ is given in \eqref{eq:X_0_inf}
for
\eqn{\label{eq:X_inf}\Xi_t=\s \Wi_t+\sum_{k=1}^{\Ni_t}\Ci-\mui \mi t}
and
\eqn{\label{eq:X_0}\Xo_t=\s \Wo_t+rt+\sum_{k=1}^{\No_t}\Co-\muo \mo t,}
where $\Wi_t$ and $\Wo_t$ are two independent copies of the standard Brownian motion. Hence the original Brownian motion is the sum $W_t=\Wi_{t \wedge \th}+\Wo_{(t-\th)^+}$.

Let $\{\mathcal{F}_t^{X}\}_{\{t\geq 0\}}$ be the natural filtration of $X$
satisfying the usual conditions such that $\forall_{t\geq 0}\; \mathcal{F}_t^X\subset\mathcal{F}$.
In the problem of the quickest detection we are looking for an optimal $\mathcal{F}_t^{X}$-stopping time $\tau^*$
that minimizes certain optimality criterion.
This criterion incorporates both the probability of false alarm $\Px(\tau<\th)$ and the mean delay time $\Ex[(\tau-\th)^+]$.
The superscript $x$ indicates the mass at zero of $\th$ defined in \eqref{eq:defpi}. Hence our problem can be stated as follows:
\begin{Problem}\label{Prob:crit1}
For each $c>0$ calculate the optimal value function
\eqn{\label{eq:crit1}V^*(x)=\inf_{\tau}\{\Px(\tau<\th) + c\Ex[(\tau-\th)^+]\}}
and find the optimal stopping time $\tau^*$ for which above infimum is attained.
\end{Problem}

The key role in solving this problem plays \textit{a posteriori} probability defined as follows:
\eqn{\label{def:pit}\pi_t:=\Px\left(\th \leq t|\FF_t^X\right).}
Note that $\pi_0=x$.
Using this \textit{a posteriori} probability, one can reformulate criterion (\ref{eq:crit1}) in the equivalent form.
\begin{Lemma}\label{lem:crit}
Criterion given in (\ref{eq:crit1}) is equivalent to:
\eqn{\label{eq:crit2}V^*(x)=\inf_{\tau}\Ex\left[1-\pi_{\tau} + c\int_0^{\tau}\pi_s\d s\right].\nonumber}
\end{Lemma}

The proof of above equivalence is given in the Appendix. Thus from now on we focus on the following optimization problem.
\begin{Problem}\label{Prob:crit2}
For given $c>0$ find the optimal value function
\eqn{\label{eq:opt_fun}V^*(x)=\inf_{\tau}\Ex\left[1-\pi_{\tau}+c\int_0^{\tau}\pi_s\d s\right]\nonumber}
and the optimal stopping time $\tau^*$ such that
\eqn{\label{eq:opt_time}V^*(x)=\Ex\left[1-\pi_{\tau^*}+c\int_0^{\tau^*}\pi_s\d s\right],\nonumber}
where $\Ex$ means expectation with respect to $\Px$, i.e. given that $\pi_0=x$.
\end{Problem}

This is a Mayer-Lagrange optimal stopping problem that, using general optimal stopping theory, could be transferred into the boundary value problem;
see Peskir and Shiryaev \cite[Chap. VI.22]{peskir2002solving} for details.
We will also follow this idea. To formulate this free-boundary problem we have to introduce additional notations.

Let $\Pxs:={\rm Law}(X|\th=s)$, for $s\in[0,\infty]$, be a family of probability measures on $(\Omega,\FF)$ such that under
the measure $\Pxs$ the drift change of the process $X_t$ is fixed and equal to $s$.
In particular, under $\Pxi$ drift change never occurs and under $\Pxo$ the drift $r\neq 0$ is present right from the beginning.
Observe that under these two measures the process $X_t$ is a L\'evy process. In this paper we assume that measures $\Pxo$ and $\Pxi$ are related by a certain change of measure introduced below.

Let
\eqn{\label{eq:lap-inf}\psi_{\infty}(\beta):=\log\Exi\left[e^{\beta X_1}\right]=\frac{1}{2}\beta^2\s^2-\beta\mui \mi -\int(1-e^{\beta y})\mui \d \Fi(y)}
be a Laplace exponent of $X_t$ under $\Pxi$. We relate $\Pxo$ and $\Pxi$ via the following change of measure:
\eqn{\label{eq:Ldef}\left.\frac{\d\Pxo}{\d\Pxi}\right|_{\FF_t^X}=L_t:=e^{\beta_0 X_t-\psi_{\infty}(\beta_0)t},}
for
\eqn{\label{eq:beta0}\beta_0:=\frac{r+\mui \mi -\muo \mo }{\s^2},}
where we assume that
\eqn{\label{eq:Lfinite}\int_{|y|\geq 1}e^{\beta_0 y}\d\Fi(y)<\infty.\nonumber}

From [Thm. 3.9. in \cite{kyprianou2006introductory}] it follows that the jump distributions of processes $\Xi_t$ and $\Xo_t$ are related with each other in the following way:
\eqn{\label{eq:jumps}
\muo=\mui\int_{\R}e^{\beta_0 y}\d\Fi(y), \q \d\Fo(y)=e^{\beta_0 y}\d\Fi(y) / \int_{\R}e^{\beta_0 y}\d\Fi(y).}
Further, the volatilities of $\Xi_t$ and $\Xo_t$ are the same and equal to $\s$. Finally, we have chosen $\beta_0$ in such a way that the drift of the process $X_t$ under $\Pxo$ is $r$, that is $\Exo X_1=\Ex \Xo_1=r$ [cf. Thm. 3.9. in \cite{kyprianou2006introductory}].

Original measure $\Px$ can be represented by combination of measures $\Pxs,s\geq 0$, as follows:
\eqn{\label{eq:P}\begin{split}\Px(\cdot)&=\Px(\cdot|\th=0)\Px(\th=0)+\Px(\cdot|\th>0)\Px(\th>0) \\&= x\Pxo(\cdot)+(1-x)\int_{\R_+}\Pxs(\cdot)\lambda e^{-\lambda s}\d s.\end{split}\nonumber}
Using the Bayes formula (see Shiryaev \cite{shiryaev2006disorder} for details)
we can represent the process $\pi_t$ in the following way:
\eqn{\label{eq:pit}\pi_t=x\left.\frac{\d\Pxo}{\d\Px}\right|_{\FF_t^X}+(1-x)\int_0^t \left.\frac{\d\Pxs}{\d\Px}\right|_{\FF_s^X}\lambda e^{-\lambda s}\d s.}
Moreover, we have:
\eqn{\label{eq:1pit}1-\pi_t=(1-x)e^{-\lambda t}\left.\frac{\d\Pxt}{\d\Px}\right|_{\FF_t^X}.}
Further, by
\eqn{\label{eq:phi1}\phi_t:=\frac{\pi_t}{1-\pi_t}}
we denote the likelihood ratio process.
Note that on $\FF_t^X$ we have $L_t=\left.\frac{\d\Pxo}{\d\Pxi}\right|_{\FF_t^X}=\left.\frac{\d\Pxo}{\d\Pxt}\right|_{\FF_t^X}$ (since we consider only events up to time $t$).
Following Shiryaev \cite{shiryaev2006disorder} and using equations (\ref{eq:pit}) and (\ref{eq:1pit}) we get that
\eqn{\label{eq:phi2}\phi_t=e^{\lambda t}L_t\left(\phi_0+\lambda\int_0^t\frac{e^{-\lambda s}}{L_s}\d s\right),}
where $\phi_0=\frac{x}{1-x}$ (for $x$ defined in (\ref{eq:defpi})).
Later we will use the representation (\ref{eq:phi2}) to the construction of the Generalized Shiryaev-Roberts statistics.

The representation (\ref{eq:phi2}) is equivalent to the certain stochastic differential equation for $\phi_t$ given in \eqref{eq:dphi} below,
which allows us to identify via relation (\ref{eq:phi1}) the infinitesimal generator $\mathcal{A}$ of the process $\pi_t$:
\eqn{\label{eq:generator}\begin{split}
\mathcal{A}&f(x)=f'(x)\left(\lambda(1-x)+x(1-x)(\mui-\muo)\right)+f''(x)\sbb x^2(1-x)^2
\\&+\int_{\R}\left[f\left(\frac{xe^{\beta_0 y}}{1+x(e^{\beta_0 y}-1)}\right)-f(x)\right]\cdot\left[(1-x)\mui \d \Fi(y)+x\muo \d \Fo(y)\right]
\end{split}}
acting on $f\in\mathcal{C}^2(\R)$.
Proof of this fact is given in Lemma \ref{lem:generator} in the Appendix.

Following the general theory of optimal stopping and free boundary problems, we can prove the following main result of this paper.
\begin{Thm}\label{thm:stopping}
Consider the following boundary value problem:
\eqn{\label{optimal_system}\begin{split}
\mathcal{A}f(x)=-cx, \quad 0\leq x<A^*,\\
f(x)=1-x, \quad A^*\leq x \leq 1,
\end{split}}
with the boundary conditions:
\eqn{\label{cont_fit}f(A^{*-})=1-A^* \quad (continuous\; fit),}
\eqn{\label{smooth_fit}f'(A^{*-})=-1 \quad (smooth\; fit),}
\eqn{\label{norm_ent}f'(0^+)=0 \quad (normal\; entrance),}
where $\mathcal{A}$ is the generator of process $\pi_t$ given by (\ref{eq:generator}).
Then the optimal value function $V^*(x)$ for Problem \ref{Prob:crit2} (hence also for Problem \ref{Prob:crit1}) solves above system for the unique point $A^*\in[0,1]$ and the optimal stopping rule for Problem \ref{Prob:crit2} (hence also for Problem \ref{Prob:crit1}) is given by:
\eqn{\label{tau_optimal}\tau^*=\inf\{t\geq 0:\pi_t\geq A^*\}.}
\end{Thm}

The proof of Theorem \ref{thm:stopping} is given in the Appendix. Above theorem identifies the optimal moment when one should "raise the alarm" that the drift has changed.
It is the first moment when \textit{a posteriori} probability $\pi_t$
exceeds certain (known) threshold $A^*$. This crucial observation leads to the construction of the Generalized Shiryaev-Roberts
statistics which we will use in the numerical analysis of the force of mortality.

\section{Generalized Shiryaev-Roberts statistics}\label{sec:Cal_USR}
The construction of classical Shiryaev-Roberts (SR) statistics is in detail described and analysed e.g.
by Shiryaev \cite{shiryaev2002quickest}, Pollak and Tartakovsky \cite{pollak2009optimality} and Moustakides et al. \cite{moustakides2009numerical}.
Following  Shiryaev \cite{shiryaev2002quickest} the classical SR statistics can be written as:
\eqn{\label{eq:classical_SR}\psi_t:=\int_0^t\frac{L_t}{L_s}\d s.\nonumber}
Observe that it can be derived from $\phi_t$ given in (\ref{eq:phi2}) in the following way.
Assume that $\pi_0=0$. Then:
\eqn{\label{eq:SR_from_USR}\lim_{\lambda\to 0}\frac{\phi_t}{\lambda}=\int_0^t\frac{L_t}{L_s}\d s=\psi_t.\nonumber}
There is one generalization of SR statistics used by Zhitlukhin and Shiryaev \cite{zhitlukhin2013bayesian}, but this is not the one consider in this paper.
Here, we start the construction of Generalized Shiryaev-Roberts statistics (GSR) from the continuous time
process $\phi_t$.
The procedure is based on the observed data $X_{t_n}$
given in discrete moments of time
$0=t_0<t_1<\ldots<t_n$, where $n$ is some fixed integer number.
Note that by passing under $\Pxi$ it is a homogeneous random walk.
To simplify our considerations, we assume that $t_i-t_{i-1}=1$. Then, recalling (\ref{eq:Ldef}), 
 we have:
\eqn{\label{eq:L_discrete}L_n=L_{t_n}=e^{\beta_0 X_n-\psi_{\infty}(\beta_0)n}=\prod_{i=1}^ne^{\beta_0(X_i-X_{i-1})-\psi_{\infty}(\beta_0)}.\nonumber}
We denote consecutive increments by $x_i:=X_i-X_{i-1}$.
Using a discrete analogue of (\ref{eq:phi2}), we define GSR statistics:
\eqn{\label{eq:phi_discrete}\gsr_n:=e^{\lambda n}L_n\left(\gsr_0+\lambda\sum_{i=0}^{n-1}\frac{e^{-\lambda i}}{L_i}\right).}
Note that GSR statistics $\gsr_n$ can be calculated using the following recursive formula:
\eqn{\label{eq:phi_recursive}\gsr_{n+1}=e^{\lambda+\beta_0x_{n+1}-\psi_{\infty}(\beta_0)}(\gsr_n+\lambda),\quad n=1,2,\ldots,}
for given $\gsr_0=\frac{x}{1-x}$. This recursive form is very convenient for further calculations.

From Theorem \ref{thm:stopping} we know that the optimal stopping rule $\tau^*$ in our detection problem is the first time when \textit{a posteriori} probability $\pi_t$ exceeds the optimal threshold
$A^*$. From (\ref{eq:phi1}) it follows that $\tau^*$ is equivalent to the first moment when $\phi_t$ exceeds threshold $B^*:=A^*/(1-A^*)$.
This means that in terms of GSR statistics for {\it the optimal alarm time} we can choose:
\eqn{\label{alarmtime}\tau^*=\min\left\{n\in\N: \gsr_n\geq \frac{A^*}{1-A^*}\right\}.}
Above stopping rule is used in the numerical analysis described in Section \ref{sec:numerical}.

We emphasize that GSR statistics is more appropriate in longevity modelling analysed in this paper than SR one.
Indeed, as we observed in \eqref{eq:SR_from_USR}, SR statistics is equivalent to GSR statistics when the parameter $\lambda$ of the exponential distribution of the moment $\theta$ of the drift change
tends to $0$. The latter case corresponds to passing with mean value of change point $\th$ to $\infty$ and hence  it becomes \textit{conditionally uniform}, see e.g. Shiryaev \cite{shiryaev2002quickest}. Still, in longevity modelling it is more likely that life tables will need to be revised in $20$ years rather than after $100$ years and
therefore keeping dependence on $\lambda>0$ in our statistics seems to be much more appropriate.

\section{Numerical analysis of longevity}\label{sec:numerical}
\subsection{Model of the force of mortality}
In the numerical part of this paper we focus on modelling the drift change of force of mortality.
Detection problem seems to be new in the context of actuarial science and we believe it
can give a new insight into how the mortality of given population is changing in time.
Mortality data is also of the capital importance for policy-making and public planning because of the public pension systems.

The main process of interest in this section is so-called force of mortality $\mu_t$, which is a
hazard rate function of the statistical length of a person at age $\omega$, say, from given population (we assume that population is homogeneous).
For example, if we choose year $1990$ to be our {\it beginning of observation, that is $t=0$}, and we fix $\omega=50$ for life tables of men, then
$\mu_1$ denotes the force of mortality for $50$-year old man in $1991$ and $\mu_{10}$ denotes
the force of mortality of $50$-year old man in $2000$.
In our model we take logarithm of the force of mortality
and we separate a deterministic part from a stochastic one. Precisely, we have
\eqn{\label{log_mu}\log\mu_t=\log\bar{\mu}_t+X_t,}
where $\bar{\mu}_t$ is the deterministic (average drift) part and $X_t$ is a process of random perturbation.
We assume that this random perturbation is given
by the process \eqref{eq:X_0_inf} with two L\'evy processes $\Xi_t$ and $\Xo_t$
given in \eqref{eq:X_inf} and \eqref{eq:X_0}, respectively,
glued together at time $\theta$ when the drift
of $X_t$ changes from zero to the nonzero one.
Precisely, we calibrate model in such a way, that at time zero (under $\Pxi$)
this random perturbation has mean $0$, that is $\Exi [X_t]=0$.
It is worth to mention here that this model is very similar to the Lee-Carter model (for fixed age $\omega$, cf. \cite{lee1992modeling}):
\eqn{\label{lee_carter}\log\mu_{\omega,t}=a_{\omega}+b_{\omega} k_t+\epsilon_{\omega,t},\nonumber}
where $a_{\omega}$ is a chosen number, $k_t$ is certain univariate time series and $\epsilon$ is a random error.
However, Lee-Carter method focuses on modelling the deterministic part of the force of mortality,
while detection procedure described in this paper concerns controlling the random perturbation in time, in fact the moment when it substantially changes.

\subsection{Calibration}\label{sec:Cal}
Now we move on to explaining how we calibrate our model.
Let $\{l_{\omega}^{(t)}\}_{\omega\geq 0}$ be some life table that gives the information on how many people at age $\omega$ are alive at time $t$, starting from
some initial new born individuals $l_0^{(t)}$.
From the definition of the force of mortality we know that
$\mu_t=-\frac{\d}{\d \omega}\log l_{\omega}^{(t)}$ which for integer valued years $i=0,1,\ldots,n$
takes the following form:
\eqn{\label{lifetablesmu}\mu_i:=\log l_{\omega}^{(i)}-\log l_{\omega+1}^{(i)}=\log \frac{l_{\omega}^{(i)}}{l_{\omega+1}^{(i)}}.}
Note that in our convention $\mu_i$ means force of mortality for fixed age $\omega$ where parameter $i$ runs through consecutive years of constructing of life tables. Similarly,  $\{l_{\omega}^{(i)}\}$ is a life table for a fixed year $i$ and runs through consecutive ages $\omega$. We start from estimating the deterministic part $\log\bar{\mu}_t$ of \eqref{log_mu}. We assume that it has the following form:
\eqn{\label{log_mu_bar}\log\bar{\mu}_t=a_0+a_1t,\nonumber}
where $a_0=\log\mu_0$ is an initial value and $a_1$ is a drift per one unit of time (i.e. one year).
We denote log-increments of $\mu_i$ by $y_i$, that is,
\eqn{\label{logincrements}y_i:=\log\mu_i-\log\mu_{i-1},\quad i=1,\ldots,n.\nonumber}
Then MLE estimator produces:
\eqn{\label{hist_drift}a_1:=\frac{1}{n}\sum_{i=1}^ny_i.\nonumber}

More attention needs to be paid for calibration of the perturbation process $X_t$; see \eqref{eq:X_0_inf} and \eqref{log_mu}.
Note that sample data we have is used only for the calibration (not detecting drift change yet).
Thus at the beginning our perturbation process $X_t$ is in fact equal to $\Xi_t$ given by (\ref{eq:X_inf}). Recall that
\eqn{\Xi_t=\s \Wi_t+\sum_{k=1}^{\Ni_t}\Ci-\mui \mi t\nonumber}
and that we choose the parameters of the above process in such a way that $\Ex \Xi_t=0$ for any $t\geq 0$.

To identify the optimal threshold $A^*$ for the GSR statistics we assume additionally that the distribution of jump sizes $\Ci$ has a double-sided exponential density:
\eqn{\label{eq:f1}\d \Fi(y):=\left(p_1\frac{1}{w}e^{-y/w}\Ind{y\geq 0}+p_2\frac{1}{w}e^{y/w}\Ind{y<0}\right)\d y,}
for some constants $p_1,p_2\in[0,1]$ such that $p_1+p_2=1$ and $w\in\R^+$. We have then
\eqn{\label{eq:mo}\mi=\int_{\R}y\d \Fi(y)=(p_1-p_2)w.\nonumber}
In Figure \ref{fig:simX} there is an exemplary simulation of the trajectory of the process $\Xi_t$.
\begin{figure}[H]
\includegraphics[width=0.8\textwidth]{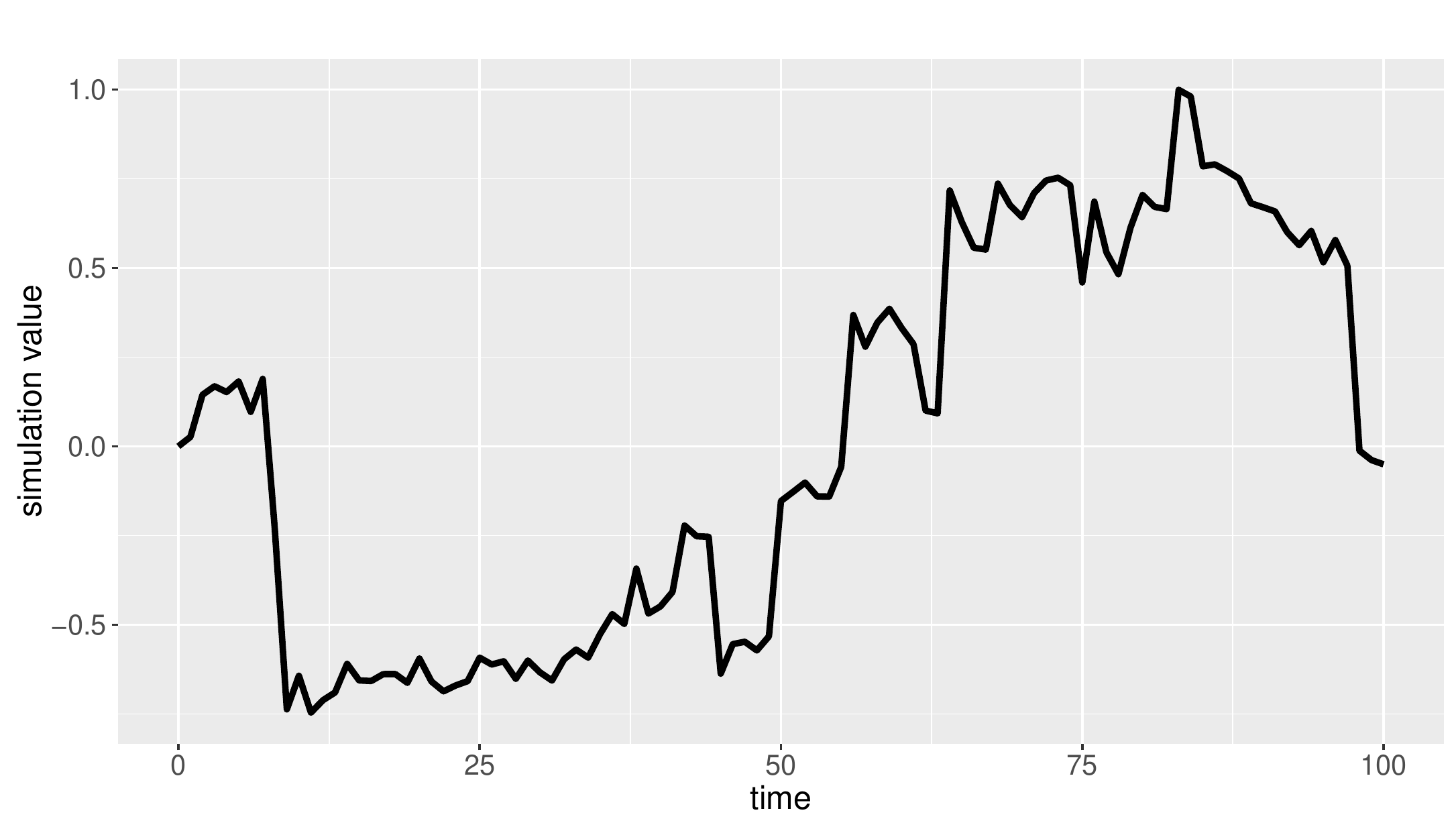}
\centering
\caption{Simulation of process $\Xi_t$.}
\label{fig:simX}
\end{figure}
From  the relation (\ref{eq:jumps}) it follows that the distribution density $\d \Fo(y)$ is also a double-sided exponential distribution:
\eqn{\label{eq:f2}\d \Fo(y)=\left(q_1\frac{1-\bb w}{w}e^{-y\frac{1-\bb w}{w}}\Ind{y\geq 0}+q_2\frac{1+\bb w}{w}e^{y\frac{1+\bb w}{w}}\Ind{y<0}\right)\d y,}
where
\eqn{\label{eq:q1q2}q_1=\frac{p_1(1+\bb w)}{1+(p_1-p_2)\bb w},\quad q_2=\frac{p_2(1-\bb w)}{1+(p_1-p_2)\bb w}.\nonumber}

We have to calibrate the volatility $\s$ of the Brownian component of the process $\Xi_t$,
the intensity $\mui$ of the Poisson process $\Ni_t$,
the mean absolute size of jump $w$ and the probability of the jumps being positive $p_1$ ($p_2$ is then given by $1-p_1$).
We denote
\eqn{x_i:=(\log\mu_i-\log\bar{\mu}_i) - (\log\mu_{i-1}-\log\bar{\mu}_{i-1}),\quad i=1,\ldots,n\nonumber}
and we proceed as follows:
\begin{enumerate}
\item\label{item:sd} we calculate $\s$ as a standard deviation of the sample $x_i$, that is $\sigma^2=\frac{1}{n-1}\sum_{i=1}^n (x_i-\overline{x})^2$ for $\overline{x}=\frac{1}{n}\sum_{i=1}^nx_i$;
\item we
check which observations from the sample $x_i$ have the absolute values greater than $z_\alpha\cdot\s$ for the $1-\alpha$-quantile of the standard Gaussian distribution,
that is, we identify which observation are outside the $(1-\alpha)$-confidence interval, later we choose a significance level of $\alpha=0.1$ and hence
$z_\alpha=1,645$;
\item we remove such observations from the sample (assuming there is at least 1 of them), treating them as jumps of the process $\Xi_t$
and we calculate $\s$ again using the first step, we mark the rejected $x_i$ increments by the subsequence $i_k$;
\item having $\s$ identified, we estimate the parameters of double exponential distribution \eqref{eq:f1} of jumps
from rejected sample $\tilde{x}_k:=x_{i_k}$ modelling jumps as follows:
$w$ is given by sample mean of absolute values of $\tilde{x}_k$,
$p_1$ as the fraction of these observations which are positive, and $\mui$ is the inverse of the sample mean of the distances $\tilde{t}_k$ between jumps
understood in the following way: $\tilde{t}_k =i_{k+1}-i_k$.
\end{enumerate}

Above there are parameters that can be estimated from the sample, but there are still some model parameters that have to chosen a priori.
In particular, we have to declare the incoming drift $r\neq 0$,
the probability $x=\Px(\theta=0)$ that the drift change occurs immediately,
the parameter $\lambda>0$ of the exponential distribution of $\th$ (conditioned that it is strictly positive) and
parameter $c$ present in optimality criterion \eqref{eq:crit1} and responsible for size of penalty for delay in raising the alarm.
These parameters have influence on the sensitivity of the detection procedure and we argue later how we choose them.

\subsection{Solving the boundary value problem}
Having all parameters of the model estimated or a priori chosen, we can apply GSR statistics as long as we manage to identify the optimal threshold $A^*$. To do this we have to solve the boundary value problem (\ref{optimal_system})- \eqref{norm_ent}. Using the same arguments as the ones given in p. 131 of Peskir and Shiryaev \citep{peskir2006optimal}, we obtain that the optimal value function $V^*(x)$ is differentiable inside the continuation set $C=(0,A^*]$, so it is in the domain of the generator $\mathcal{A}$. Here, we can use the theory of singular ordinary differential equations.
\begin{Thm}\label{thmbvp}
Assume that
\[
	\beta_0 ^2 \gamma ^2 \sigma ^2-\beta_0 ^2 \sigma ^2+2 \gamma ^2 \lambda +2 \gamma ^2 \mui -2 \gamma ^2 \muo+2 \gamma  \muo-2 \lambda -2 \mui+4 \muo-4 \gamma \muo q_2 > 0,
\]
where $\gamma:=1/(\bb w)$. The solution of the boundary value problem (\ref{optimal_system}) - \eqref{norm_ent}
is given by
\eqn{\label{v*}
	V^*(x) = 1 - A^* - \int_x^{A^*} \frac{u(z)}{1-z} \d z,
}
where $A^*$ is such that $u(A^*) = A^* - 1$. The function $u(x)$ is a solution of
\eqn{\label{eq:main4}
	\sum_{k=0}^{3} (1-x)^{k+1} b_k(x) u^{(k)}(x) = c(\gamma^2-1)x(1-x),
}
for
\[
\begin{split}
	b_k(x) &= \frac{1}{k!}\sum_{i=k}^{3} i! \; a_i(x), \\
	a_0(x) &= -\lambda\gamma^2+x\Big(\muo(x(2\gamma q_2-\gamma+1)+\gamma^2-1)\\
	&\q\q\q\q\q\q\q+\mui(x(\gamma-2\gamma p_2)+2\gamma p_2-\gamma^2-\gamma)\Big), \\
	a_1(x) &= x(\lambda-2\lambda x)+x^2\Big(\muo(3x-3)+\mui(2-3x)\\
	&\q\q\q\q\q\q\q-\sbb(\gamma^2-12x^2+15x-4)\Big), \\
	a_2(x) &= x^2 \lambda-x^3 \left(\muo-\mui+(8x-5)\sbb)\right), \\
	a_3(x) &= \sbb x^4,
\end{split}
\]
supplemented by the initial conditions
\eqn{
	\begin{split}
	&u(0) = 0, \quad u'(0) = -\frac{c}{\lambda} \\
	&u''(0) = -\frac{c \left(-\beta_0 ^2 \gamma ^2 \sigma ^2+4 \beta_0 ^2 \sigma ^2+2 \gamma ^2 \lambda -2 \gamma ^2 \mui+2 \gamma ^2 \muo-2 \gamma  \mui-8 \lambda +4 \mui-8 \muo+4 \gamma  \mui p_2\right)}{\left(\gamma ^2-4\right) \lambda ^2}.
	\end{split}
\label{eq:initial}
}
Moreover, $u(x)$ can be represented by the following \emph{asymptotic} series
\eqn{\label{solbvp}
	u(x)=\sum_{k=1}^\infty \alpha_ix^i \quad \text{as} \quad x\rightarrow 0^+,
}
with the coefficients $\{\alpha_i\}_{\{i=0\}}^\infty$ that could be derived from the formula (\ref{eq:main4}).
\end{Thm}
The proof of above theorem is given in the Appendix. The numerical procedure of finding $V^*(x)$ can be conducted in three steps. First, numerically solve (\ref{eq:main4}) in order to find $u(x)$. Then, use a root finding algorithm to compute the zero $A^*$ of the function $u(x) + 1 - x$. Finally, calculate $V^*(x)$ by (\ref{v*}).

\subsection{Drift change detection for Polish life tables}
We apply described drift change detection procedure to the analysis of the Polish life tables.
We consider Polish life tables for years from 1960 to 2014, downloaded from The Human Mortality Database \cite{mortalityDatabase}.
For fixed age $\omega$ we check how the force of mortality has been changing over these years.
We also detect, using introduced Generalized Shiryaev-Roberts statistics, the significant change of drift in mortality.
At the beginning, in Figure \ref{fig:mort45}, we give the exemplary plot of the force of mortality for Polish men at age $45$.
\begin{figure}[H]
\includegraphics[width=0.95\textwidth]{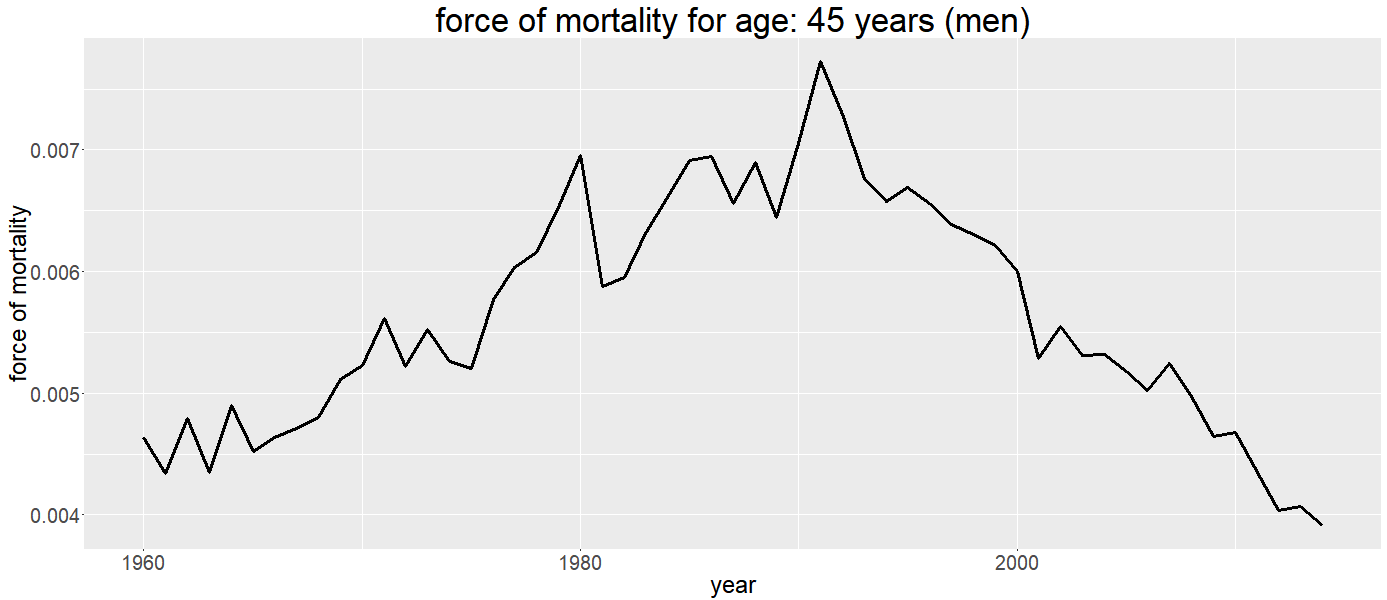}
\centering
\caption{Force of mortality for men in Poland at age 45, years 1960-2014}
\label{fig:mort45}
\end{figure}
Observe that the force of mortality is increasing in the first part of the plot and then decreasing in the second part.
To proceed with detection algorithm, we have to determine which data will be used to calibrate our model. In this example
we use for that purpose first 20 years of observations, i.e. years 1960-1980. Then, from the year 1980 we start to look for the change of drift in mortality.

There are still some parameters that need to be determined arbitrary. We assume that their values are as follows:
\begin{itemize}
\item $\lambda=0.05$ -- the parameter of the exponential distribution of $\theta$ conditioned to be strictly positive;
\item $x=\P^x(\theta=0)=0.05$;
\item $c=0.02$ -- the weight of the mean delay time inside the optimality criterion \eqref{eq:crit1};
\item drift incoming after the change time $\th$ -- we consider two values, dependent on $\s$: $r=-3\sigma$ or $r=-5\sigma$.
\end{itemize}

The result for $r=-3\s$ is shown in Figure \ref{fig:det45}.
\begin{figure}[H]
\includegraphics[width=0.95\textwidth]{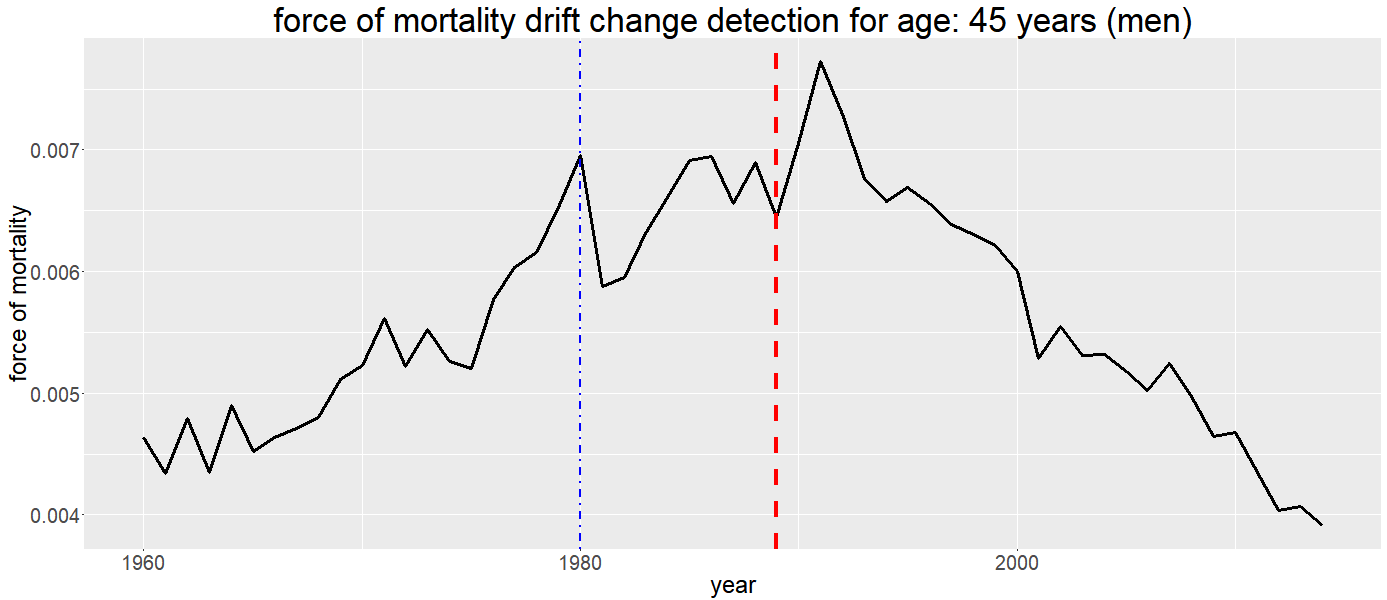}
\centering
\caption{Force of mortality drift change detection for men in Poland at age 45, years 1960-2014}
\label{fig:det45}
\end{figure}
Blue vertical line indicates year $1980$, when the detection algorithm starts. The red vertical line shows the moment of drift change detected by our procedure -- year $1989$. Indeed, we can observe that the drift stabilizes after $1989$. In general, the parameter $r$ determines how sensitive the algorithm is, so greater absolute values of $r$ cause later detections.

This relationship can be observed in Figure \ref{fig:men_all}, which consists of four smaller plots.
In the first column there are the same force of mortality plots for men at age $45$, but they differ by the red vertical line indicating moments of detection. The plot in the first row is for parameter $r=-3\s$, while the second plot is for $r=-5\s$. As we can see, in the lower plot detection occurred later.
The second column in Figure \ref{fig:men_all} presents analogous two plots, but for men at age $60$. In the first row there is again plot for parameter $r=-3\s$ and in the second one -- for $r=-5\s$. This time both detection moments are quite close to each other.
\begin{figure}[H]
\includegraphics[width=0.48\textwidth]{mortality_men_45_r3_c002_d005_p005.png}
\includegraphics[width=0.48\textwidth]{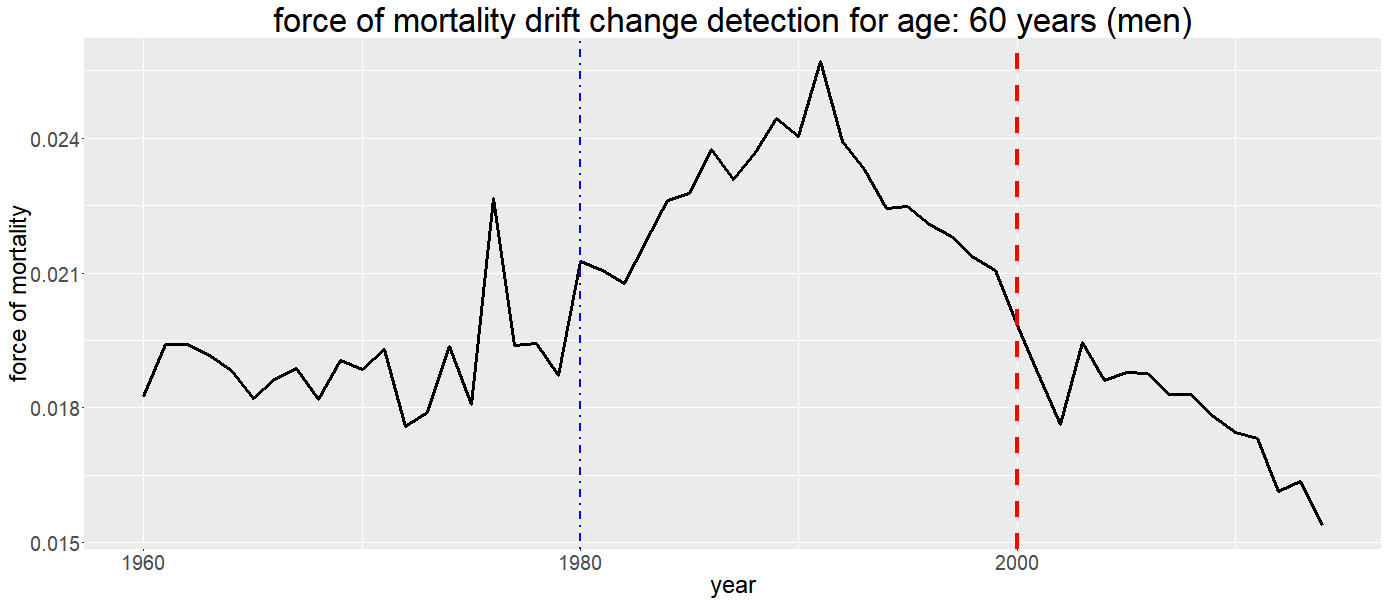}
\includegraphics[width=0.48\textwidth]{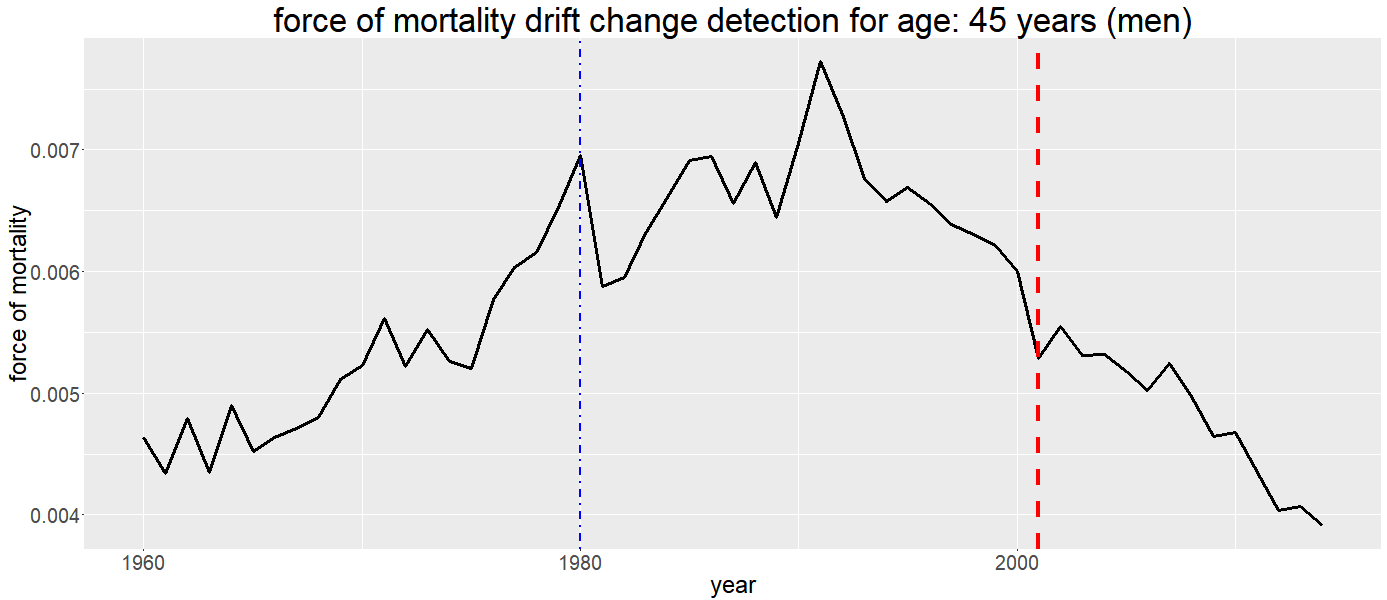}
\includegraphics[width=0.48\textwidth]{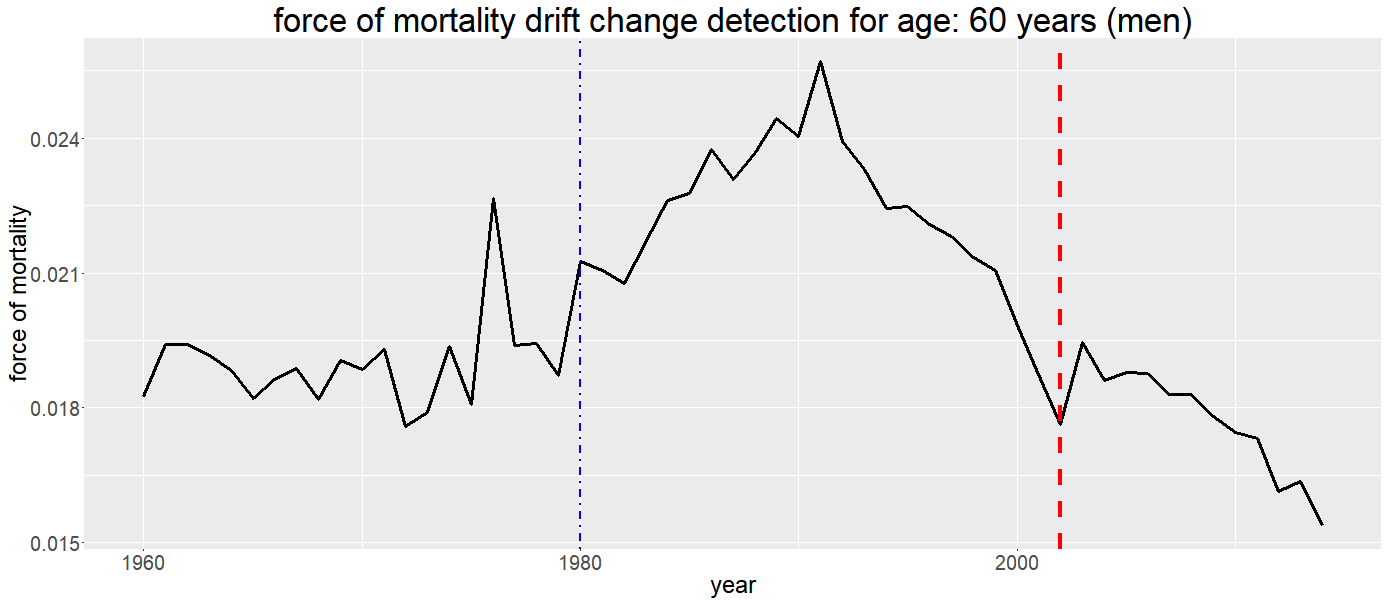}
\centering
\caption{Force of mortality drift change detection for men in Poland at age 45 and 60, $r=-3\sigma$ (first row) and $r=-5\sigma$ (second row), years 1960-2014.}
\label{fig:men_all}
\end{figure}

We can also analyse similar plots for women at age $45$ and $60$. The results are presented in Figure \ref{fig:women_all}.
Note that in both columns, on plots in the second row, the red lines are at the end -- it indicates that the drift change was not detected at all.
This shows that there might be some scenarios where we can expect the drift change (in Figure \ref{fig:women_all}, second row, it is around the year $1990$),
but our algorithm based on GSR statistics treats this change insignificant for chosen new drift $r=-5\s$.
\begin{figure}[H]
\includegraphics[width=0.48\textwidth]{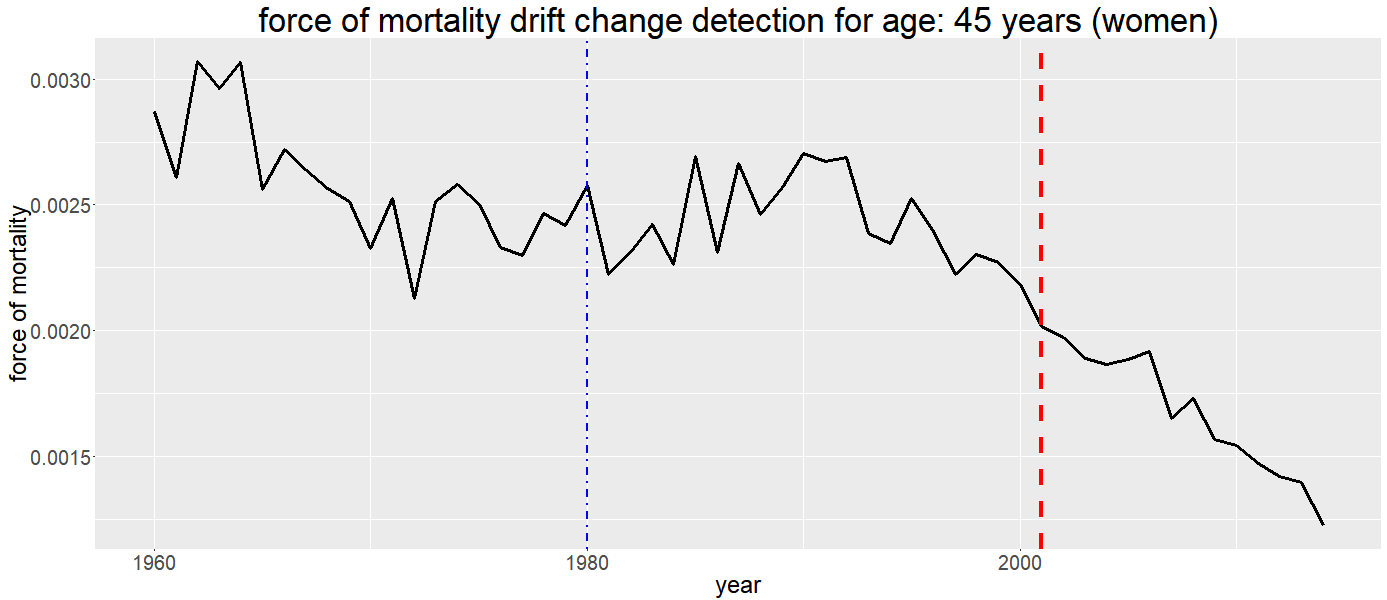}
\includegraphics[width=0.48\textwidth]{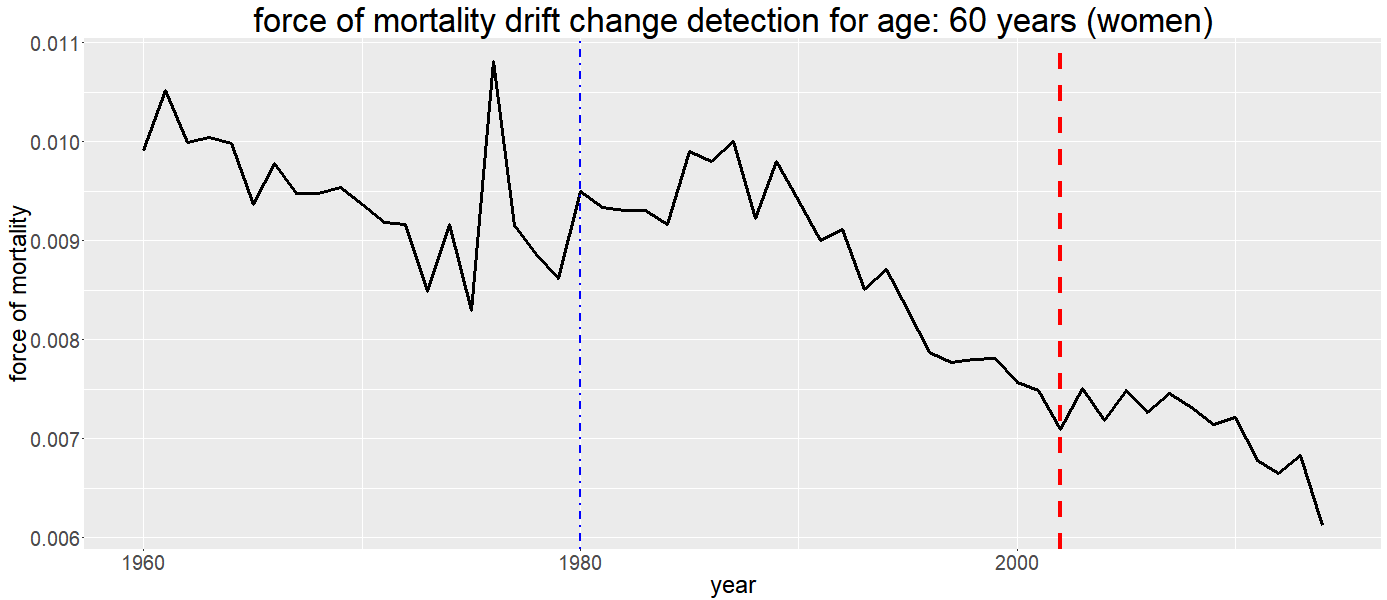}
\includegraphics[width=0.48\textwidth]{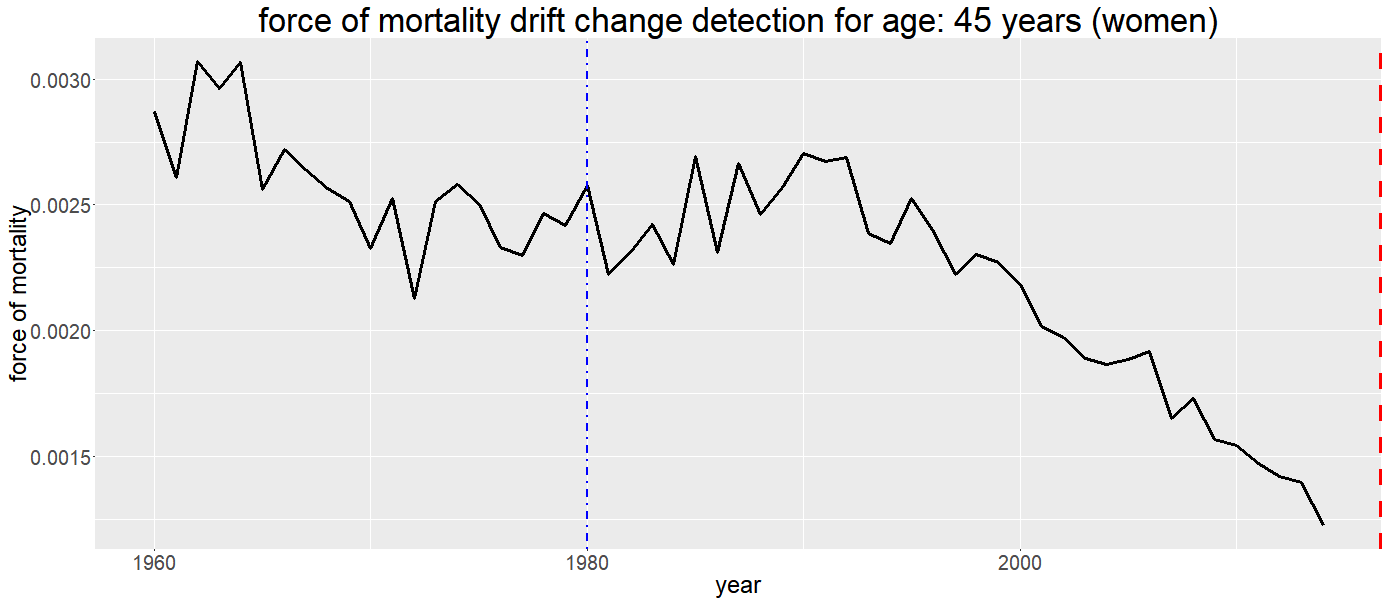}
\includegraphics[width=0.48\textwidth]{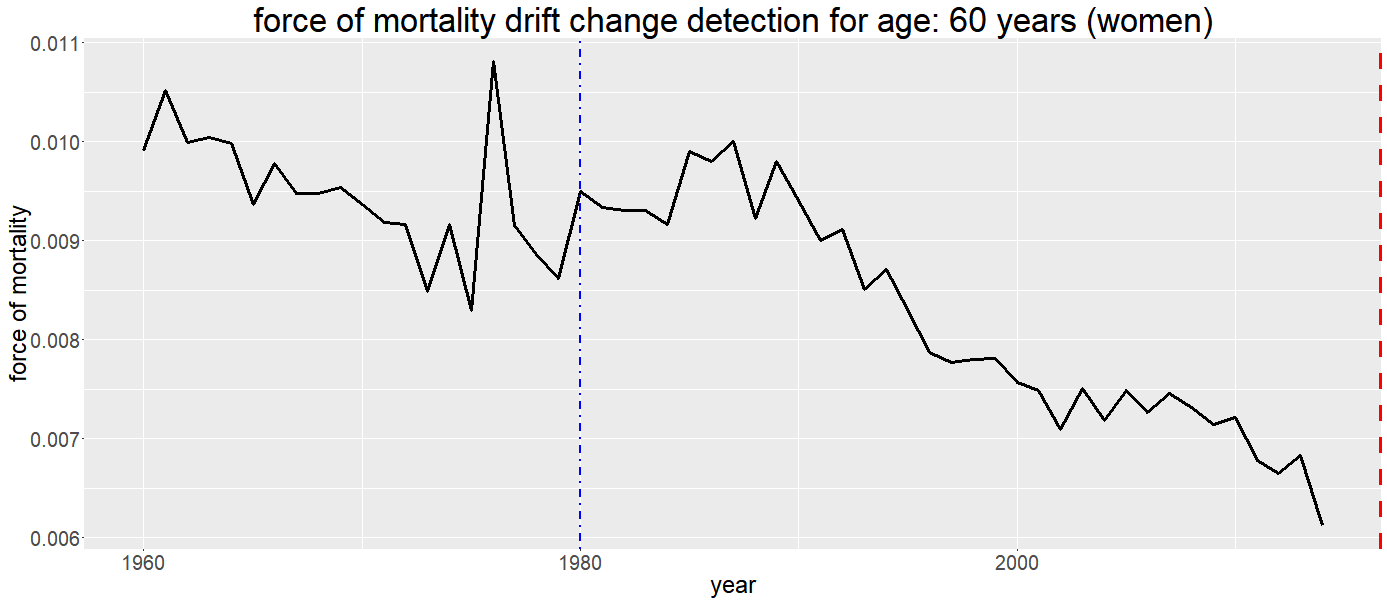}
\centering
\caption{Force of mortality drift change detection for women in Poland at age 45 and 60, $r=-3\sigma$ (first row) and $r=-5\sigma$ (second row), years 1960-2014.}
\label{fig:women_all}
\end{figure}

\section{Conclusions}\label{sec:con}
In this paper we solved the quickest drift change detection problem
for a L\'evy process consisting of both a continuous Gaussian part and a jump component.
We considered here Bayesian framework with an exponential a priori distribution of the change point
using an optimality criterion based on a probability of false alarm and an expected delay of the detection.
Using this solution we constructed the  Generalized Shiryaev-Roberts statistics and applied it
in detecting the change of the force of mortality in the Polish life tables.

It is natural to consider multivariate L\'evy processes
with dependent components for example to model the change of the force mortality of couples.
It is important to analyze this dependence since many papers about health and mortality has consistently identified that unmarried individuals generally
report a higher mortality risk than their married counterparts, with men being particularly affected in this respect.
One can also consider other a priori distribution of the change point. Unfortunately, in this case the optimal
stopping rule is much more complex and more difficult to implement; see \cite{zhitlukhin2013bayesian} in the case of the Brownian set-up.
Finally, one can consider other than \eqref{eq:Ldef} change of measure linking $\Pxo$ and $\Pxi$ as it is suggested in \cite{palmowski2002}.
This will be the subject of future research.

\section{Appendix}\label{sec:app}

\paragraph{Proof of Lemma \ref{lem:crit}}
Note that:
\eqn{\label{crit:p1}\Px(\tau<\th)=\Ex[\Ex[\Ind{\tau<\th}|\FF_{\tau}^X]]=\Ex[1-\Px(\th\leq\tau|\FF_{\tau}^X)]=\Ex[1-\pi_{\tau}].}
Moreover, observe that by  Tonelli's theorem we have:
\begin{equation}\label{crit:p2}\begin{split}
&\Ex[(\tau-\th)^+] = \int_{\R_+}\Ex[(t-\th)^+]\Px(\tau\in \d t) = \int_{\R_+}\Ex\left[\int_0^t\Ind{\th\leq s}\d s\right]\Px(\tau\in \d t) \\
&= \int_{\R_+}\int_0^t\Ex\left[\Ex\left[\Ind{\th\leq s}|\FF_s^X\right]\right]\d s\Px(\tau\in \d t) = \int_{\R_+}\int_0^t\Ex[\pi_s]\d s\Px(\tau\in \d t) \\
&= \int_{\R_+}\Ex\left[\int_0^t\pi_s\d s\right]\Px(\tau\in \d t) = \Ex\left[\int_0^{\tau}\pi_s\d s\right].
\end{split}\end{equation}
Putting together (\ref{crit:p1}) and (\ref{crit:p2}) completes the proof.
\exit

\begin{Lemma}\label{lem:generator}
The generator of the process $\pi_t$ defined in (\ref{def:pit}) is given by:
\eqn{\begin{split}
\mathcal{A}&f(x)=f'(x)\left(\lambda(1-x)+x(1-x)(\mui-\muo)\right)+f''(x)\sbb x^2(1-x)^2
\\&+\int_{\R}\left[f\left(\frac{xe^{\beta_0 y}}{1+x(e^{\beta_0 y}-1)}\right)-f(x)\right]\cdot\left[(1-x)\mui \d \Fi(y)+x\muo \d \Fo(y)\right]
\end{split}}
for $f\in\mathcal{C}^2(\R)$.
\end{Lemma}
\begin{proof}
In order to derive the generator of $\pi_t$ we will use It\'o's formula. From (\ref{eq:X}) we have
\eqn{\label{eq:dX}\d X_t=\s \d W_t+\Delta X_t+\left(-\mui \mi\Ind{\th>t}+(r-\muo \mo\Ind{\th\leq t})\right)\d t,\nonumber}
where $\Delta X_t=X_t-X_{t_-}$ refers to jumps of $X$ at time $t$. Using (\ref{eq:Ldef}) we have:
\begin{multline*}
\d L_t=\Bigg(\frac{1}{2}\beta_0^2\s^2-\psi_{\infty}(\beta_0)\\
+\beta_0\left((r-\muo \mo)\Ind{\th\leq t}-\mui \mi\Ind{\th>t}\right)\Bigg)L_t\d t\\
+ \beta_0\s L_t\d W_t + L_{t_-}\left(\frac{L_t}{L_{t_-}}-1\right).\nonumber
\end{multline*}
By assumption \eqref{eq:jumps},
\eqn{\int_\R(1-e^{\beta_0 y})\mui \d \Fi(y)=\mui-\muo\nonumber}
and hence by (\ref{eq:lap-inf}) and (\ref{eq:beta0}) we get
\eqn{\label{eq:dL}\d L_t=\left(\beta_0^2\s^2\Ind{\th\leq t}+\mui-\muo\right)L_t\d t+\beta_0\s L_t\d W_t+L_{t_-}\left(e^{\beta_0\Delta X_t}-1\right).\nonumber}
Now using (\ref{eq:phi2}) and the integration-by-parts formula for semimartingales, we derive:
\begin{multline}
\label{eq:dphi}\d\phi_t=\lambda(\phi_t+1) \d t+\left(\beta_0^2\s^2\Ind{\th\leq t}+\mui-\muo\right)\phi_t \d t \\
+ \beta_0\s \phi_t \d W_t+\phi_{t_-}\left(e^{\beta_0\Delta X_t}-1\right)
\end{multline}
and (\ref{eq:phi1}) together with It\'o's formula produces:
\begin{multline}
\label{eq:dpi1}\d\pi_t=\lambda(1-\pi_t)\d t+\pi_t(1-\pi_t)\left(\beta_0^2\s^2\Ind{\th\leq t}+\mui-\muo\right)\d t+\\
+\pi_t(1-\pi_t)\beta_0\s \d W_t-\pi_t^2(1-\pi_t)\s^2\beta_0^2 \d t + \Delta\pi_t.
\end{multline}
The jumps $\Delta\pi_t$ can be expressed in terms of process $X$ as follows:
\begin{multline*}
\Delta\pi_t=\pi_{t_-}\left(\frac{\pi_t}{\pi_{t_-}}-1\right)=\pi_{t_-}\left(\frac{\phi_t}{\phi_{t_-}}\frac{1+\phi_{t_-}}{1+\phi_t}-1\right)\\=
\pi_{t_-}\left(e^{\beta_0\Delta X_t}\frac{1+\phi_{t_-}}{1+\phi_{t_-}e^{\beta_0\Delta X_t}}-1\right)=\pi_{t_-}\left(\frac{e^{\beta_0\Delta X_t}}{1-\pi_{t_-}+\pi_{t_-}e^{\beta_0\Delta X_t}}-1\right)\\
=\frac{\pi_{t_-}(1-\pi_{t_-})\left(e^{\beta_0\Delta X_t}-1\right)}{1+\pi_{t_-}\left(e^{\beta_0\Delta X_t}-1\right)}.
\end{multline*}
To prove that $\pi_t$ solving \eqref{eq:dpi1} is a Markov process,
we introduce auxiliary process
\eqn{\bar{W}_t=\beta_0\s\int_0^t(\Ind{\th\leq s}-\pi_s)\d s+W_t,\nonumber}
which can be checked to be a square integrable martingale with respect to filtration $\FF_t^X$ such that $\E[(\bar{W}_t-\bar{W}_s)^2]=t-s$, see e.g. Shiryaev \cite{shiryaev2006disorder}.
Hence by the L\'evy's theorem it is a Brownian motion. Putting $\bar{W}_t$ into (\ref{eq:dpi1}) we obtain the final form of the dynamics of the process $\pi_t$:
\eqn{\label{eq:dpi2}\d\pi_t=\lambda(1-\pi_t)\d t+\pi_t(1-\pi_t)\beta_0\s \d\bar{W}_t+\pi_t(1-\pi_t)(\mui-\muo)\d t+\Delta\pi_t.}
Thus indeed, $\pi_t$ is a Markov process.

Using It\'o's lemma one more time for $f\in \mathcal{C}^2(\R)$ and using (\ref{eq:dpi2}) together with Dynkin formula,
we can find out that the generator of the process $\pi_t$ is given by:
\eqn{\begin{split}
\mathcal{A}&f(x)=f'(x)\left(\lambda(1-x)+x(1-x)(\mui-\muo)\right)+f''(x)\sbb x^2(1-x)^2
\\&+\int_{\R}\left[f\left(\frac{xe^{\beta_0 y}}{1+x(e^{\beta_0 y}-1)}\right)-f(x)\right]\cdot\left[(1-x)\mui \d \Fi(y)+x\muo \d \Fo(y)\right],
\end{split}\nonumber}
which completes the proof.
\end{proof}

\begin{Lemma}\label{lem:optimal_system}
The optimal value function $V^*(x)$ in Problem \ref{Prob:crit2} satisfies the system
\eqn{\label{eq:sys_lem}\left\{\begin{array}{ll}\mathcal{A}V^*(x)=-cx,&x\in C,\\ V^*(x)=1-x,&x\in D,\end{array}\right.}
where $\mathcal{A}$ is the generator of process $\pi_t$ given in (\ref{eq:generator}), $C$ is an open continuation set
and $D = C^c$ is a stopping set. That is, the optimal stopping rule for Problem \ref{Prob:crit2} is given by
\eqn{\tau^*=\inf\{t\geq 0:\pi_t\in D\}.\nonumber}
\end{Lemma}
\begin{proof}
From \cite[Sec. 7.1 and 7.2, Chap. III, p. 130-133]{peskir2006optimal} it follows that the function $V^*(x)$ solves the sum of \textit{Dirichlet} and \textit{Dirichlet/Poisson} problems and hence it satisfies system (\ref{eq:sys_lem}). Theorem 2.4, Chap. I on p. 37 in \cite{peskir2006optimal} indicates that the first entry time into stopping set $D$ is optimal for this problem.
\end{proof}

\begin{Lemma}\label{lem:concave}
The optimal value function $V^*(x)$ in problem \ref{Prob:crit2} is concave.
\end{Lemma}
\begin{proof}
Let us denote
\eqn{\label{eq:g_fun}G(x,t):=\Ex\left[1-\pi_t+c\int_0^t\pi_s\d s\right].}
We can observe that function $G(x,t)$ is continuous with respect to $t$ for all $x\in[0,1]$ and linear with respect to $x$ for any $t\geq 0$. Indeed,
\eqn{\Ex[\pi_t]=\Px(\th\leq t)=x+(1-x)(1-e^{-\lambda t})=xe^{-\lambda t}+(1-e^{-\lambda t})\nonumber}
and
\eqn{\Ex\left[c\int_0^t\pi_s\d s\right]=c\int_0^t\Ex[\pi_s]\d s=cx\int_0^te^{-\lambda s}\d s+c\int_0^t(1-e^{-\lambda s})\d s.\nonumber}
Therefore, function $G(x,t)$ is linear with respect to $x$ as a sum of linear functions. The optimal value function can be expressed as
\eqn{V^*(x)=\inf_{\tau}G(x,\tau).\nonumber}
To prove that it is concave, firstly consider only stopping times $\tau\leq T$ for some finite time horizon $T$. For some fixed $N\in\N^+$ let $A_N:=\{0,T/N, 2T/N,\ldots,T\}$. Consider smaller stopping times families $\mathcal{M}_{k,N}:=\{\tau\leq T:\tau\in A_N, \tau\geq kT/N\}$ for $k\in\{0,1,\ldots,N\}$. Further, let us define
\eqn{V_N^T(x,k):=\inf_{\tau\in\mathcal{M}_{k,N}}G(x,\tau).\nonumber}
Using the principle of dynamic programming we get the following equalities:
\eqn{V_N^T(x,N)=G(x,T),\nonumber}
\eqn{V_N^T(x,k)=\min\left\{V_N^T(x,k+1), G(x,kT/N)\}, \; k\in\{0,1,\ldots,N-1\right\}.\nonumber}
Since $G(x,t)$ is concave (because it is linear) w.r.t. $x$ and minimum of two concave functions is again concave, we conclude that $V_N^T(x,0)$ is concave.

Now we prove that $V_N^T(x,0)\xrightarrow{{N\to\infty}}V^T(x):=\inf_{\tau\leq T}G(x,\tau)$. Let $\tau^{*T}$ be optimal for $V^T(x)$, i.e. $V^T(x)=G(x,\tau^{*T})$. Consider $\bar{\tau}_N^T:=\inf\{u\geq \tau^{*T}:u\in A_N\}$. Then $\bar{\tau}_N^T\in\mathcal{M}_{0,N}$ and the following inequality holds
\eqn{\label{eq:lem_conc_ineq}G\left(x,\bar{\tau}_N^T\right)\geq V_N^T(x,0)\geq V^T(x).}
But $\bar{\tau}_N^T\xrightarrow{N\to\infty}\tau^{T*}$ a.s. from the definition. Since $G(x,t)$ is continuous w.r.t. $t$, then
\eqn{G\left(x,\bar{\tau}^T_N\right)\xrightarrow{N\to\infty}G\left(x,\tau^{*T}\right)=V^T(x).\nonumber}
Hence, using inequality (\ref{eq:lem_conc_ineq}) we get that $V_N^T(x,0)\xrightarrow{N\to\infty}V^T(x)$ for all $x\in[0,1]$. Since the limit of a convergent sequence of concave functions is again concave, we conclude that $V^T(x)=\inf_{\tau\leq T}G(x,\tau)$ is concave. Now, passing to infinity with time horizon $T$, we obtain
\eqn{V^T(x)\xrightarrow{T\to\infty}V^*(x).\nonumber}
We can now conclude that $V^*(x)$ is indeed concave as the limit of the sequence of concave functions and the proof is completed.
\end{proof}

\begin{Lemma}\label{lem:contset}
The continuation set is equal to $C=[0,A^*)$.
\end{Lemma}
\begin{proof}
Let us observe that $V^*(x)$ is bounded from above by $V_0(x):=1-x$ for all $x\in[0,1]$. Consider the stopping time $\tau_0\equiv 0$. Then $G(x,\tau_0)=1-x$ for $G(x,t)$ given by (\ref{eq:g_fun}). Since $V^*(x)=\inf_{\tau}G(x,\tau)$, we get that indeed $V^*(x)\leq G(x,\tau_0)=1-x$.

Since $V^*(x)$ is concave, then the continuation set is either of the form $C=[0,A) \cup (B,1]$ or $[0,A)$, for some $A,B\in[0,1]$. Now, if the first case holds, then $V^*(1)<0$, which contradicts the definition of the value function which is nonnegative. Hence, $C=[0,A^*)$ for some $A^*\in[0,1]$.
\end{proof}

\begin{Lemma}\label{lem:nonincr}
The optimal value function $V^*(x)$ in Problem \ref{Prob:crit2} satisfies the \textit{normal entrance} boundary condition (\ref{norm_ent}) and it is non-increasing.
\end{Lemma}
\begin{proof}
We start from the proof that $0\in C$. Consider $x=0$. Then the drift always changes at strictly positive, exponentially distributed time $\th$. Therefore it is not optimal to stop immediately. Thus $A^*>0$ and $V^*(0)<1$.

From Lemmas \ref{lem:optimal_system} and \ref{lem:contset} it follows that for some $A^*>0$ the optimal value function $V^*$ satisfies the following system of equations:
\eqn{\label{eq:sys_lem2}\left\{\begin{array}{ll}\mathcal{A}V^*(x)=-cx,&x\in [0,A^*),\\ V^*(x)=1-x,&x\in [A^*,1],\end{array}\right.}
for $\mathcal{A}$ given in (\ref{eq:generator}). Taking $x \to 0$ in the first equation of (\ref{eq:sys_lem2}) we get the \textit{normal entrance} condition ${V^*}'(0^+)=0$.
Now from Lemma \ref{lem:concave} we know that $V^*(x)$ is concave and hence ${V^*}''(x)\leq 0$. This means that ${V^*}'(x)$ is non-increasing and, since ${V^*}'(0^+)=0$, then $\forall_{x\in[0,1]} \; {V^*}'(x)\leq 0$. Hence $V^*(x)$ is non-increasing.
\end{proof}

\begin{Lemma}\label{smoothfitlemma}
The optimal value function $V^*(x)$ in Problem \ref{Prob:crit2} satisfies the smooth fit condition
\eqref{smooth_fit}.
\end{Lemma}
\begin{proof}
From the general optimal stopping theory we know that the payoff function $1-x$ dominates $V^*(x)$
and $V^*(\pi_t)-c\pi_t$ is a submartingale; see Peskir and Shiryaev \cite[Thm. 24, p. 37, and Chap.III]{peskir2006optimal}.

Since the payoff function dominates the value function
and both are non-increasing we have that
${V^*}'(A^*-)\geq {V^*}'(A^*+)$.
To prove the inequality in the opposite direction we
use the change of variable formula presented in Eisenbaum and Kyprianou \cite{Eisenbaum&Kyprianou} together with Dynkin formula:
\begin{align}
V^*(\pi_t)-c\pi_t=&M_t+V^*(x)-cx+\int_0^t\mathcal{A}(V^*_1(\pi_s)-c\pi_s)\d s +\int_0^t\mathcal{A}(V^*_2(\pi_s)-c\pi_s)\d s\nonumber\\
&+ \int_0^t\frac{\partial}{\partial x}\left(V^*(\pi_{s+})-V^*(\pi_{s-})\right)\d L^{A^*}_s,\nonumber
\end{align}
where $M_t$ is a local martingale, $L^{A^*}_s$ is a local time of $\pi$ at $A^*$ and $V^*_1(x)=V^*(x)|_{x>A^*}$ and $V^*_2(x)=V^*(x)|_{x<A^*}$.
Note that functions $V^*_1$ and $V^*_2$ are in the domain of the infinitesimal generator $\mathcal{A}$ (see Eisenbaum and Kyprianou \cite[Thm. 2]{Eisenbaum&Kyprianou}).
Now, from the fact that $V^*(\pi_t)-c\pi_t$ is a submartingale it follows that
\begin{eqnarray*}
\lefteqn{\Ex\left\{\int_w^t\mathcal{A(}V^*_1(\pi_s,)-c\pi_s)\d s+\int_w^t\mathcal{A}(V^*_2(\pi_s)-c\pi_s)\d s\right.} \\&&\left.+\int_w^t\frac{\partial}{\partial x}\left(V^*(\pi_{s+})-V^*(\pi_{s-})\right)\d L^{A^*}_s\right\}\geq 0\label{inequ1}
\end{eqnarray*}
for any $0\leq w\leq t$.
From Eisenbaum and Kyprianou \cite[Thm. 3]{Eisenbaum&Kyprianou} it follows that the process
$$t\rightarrow \int_w^t\frac{\partial}{\partial x}\left(V^*(\pi_{s+})-V^*(\pi_{s-})\right)\d L^{A^*}_s$$ is of unbounded variation on any finite interval
similarly as $X_t$ is by assumption \eqref{sigma}.
Additionally, the processes
$t\rightarrow \int_0^t\mathcal{A}(V^*_1(\pi_s)-c\pi_s)\d s $ and $t\rightarrow \int_0^t\mathcal{A}(V^*_2(\pi_s)-c\pi_s)\d s $ are of bounded variation.
Thus, taking $t\rightarrow 0$ in \eqref{inequ1} we can conclude that
\begin{align}
\Ex\int_w^te^{-qs}\frac{\partial}{\partial x}\left(V^*(\pi_{s+})-V^*(\pi_{s-})\right)\d L^{A^*}_s\geq 0
\end{align}
for all sufficiently small $w$ and $t$, otherwise dividing \eqref{inequ1} by $t-w$ and taking $w\rightarrow t$ would produce a contradiction.
Since the local time $L_t^{A^*}$ is nondecreasing and it increases only when process $\pi_t$ enters interval $(A^*,1]$ from the set $[0, A^*)$,
then by taking $x=A^*$, the following inequality must hold true:
$$\frac{\partial}{\partial x}V^*(x)|_{x\downarrow A^*}-\frac{\partial}{\partial x}V^*(x)|_{x\uparrow A^*}\geq 0.$$
This inequality
completes the proof of the \textit{smooth fit} property at $A^*$.
\end{proof}

\paragraph{Proof of Theorem \ref{thm:stopping}}
From Lemmas \ref{lem:optimal_system} and \ref{lem:contset} it follows that the optimal value function $V^*(x)$ satisfies the system (\ref{optimal_system}) for some $A^*$. From Lemmas \ref{smoothfitlemma}
and \ref{lem:nonincr} we know that it satisfies boundary conditions \eqref{smooth_fit} and (\ref{norm_ent}). The boundary condition (\ref{cont_fit}) is satisfied just by the definition of the value function, which is continuous.

\exit

\paragraph{Proof of Theorem \ref{thmbvp}}
We recall that the infinitesimal generator $\mathcal{A}$ of $\pi_t$ appearing in (\ref{optimal_system}) is given by (\ref{eq:generator}), where
the distributions of jump sizes $\Fi(y)$ and $\Fo(y)$ are given by (\ref{eq:f1}) and \eqref{eq:f2}, respectively.

We are looking for the function $f(x)$ solving the equation
\eqn{\label{eq:mainf}\mathcal{A}f(x)=-cx, \quad x\in[0,A^*),}
given that
\eqn{\label{eq:boundf}f(A^{*-})=1-A^*,\quad f'(A^{*-})=-1,\quad f'(0^+)=0.}
For chosen jump distribution (\ref{eq:f1}), equation (\ref{eq:mainf}) takes the following form:
\eqn{\label{eq:main1}\begin{split}
&f'(x)\left(\lambda(1-x)-x(1-x)(\muo-\mui)\right)+f''(x)\sbb x^2(1-x)^2+cx\\
&+\int_{\R}\left[f\left(\frac{xe^{\bb y}}{1+x(e^{\bb y}-1)}\right)-f(x)\right]\\
&\q\cdot \Bigg[\left(\frac{(1-x)\mui p_1}{w}e^{-y/w}+\frac{x\muo q_1(1-\bb w)}{w}e^{-y\frac{1-\bb w}{w}}\right)\Ind{y\geq0}\\
&\q+\left(\frac{(1-x)\mui p_2}{w}e^{y/w}+\frac{x\muo q_2(1+\bb w)}{w}e^{y\frac{1+\bb w}{w}}\right)\Ind{y<0}\Bigg]dy=0.
\end{split}}
The integral in the above equation can be divided into two disjoint integration regions: from $-\infty$ to 0 and from 0 to $\infty$.
Then both of these integrals can be integrated by parts. Further, we substitute $z=xe^{\bb y}/(1+x(e^{\bb y}-1))$. Using the following observations
\eqn{\frac{\d}{\d y}\left(\frac{xe^{\bb y}}{1+x(e^{\bb y}-1)}\right)=-\frac{\bb(x-1)xe^{\bb y}}{(x(e^{\bb y}-1)+1)^2},\nonumber}
\eqn{\begin{split}&\lim_{y\to\infty}\frac{xe^{\bb y}}{1+x(e^{\bb y}-1)}=1,\\ &\lim_{y\to 0}\frac{xe^{\bb y}}{1+x(e^{\bb y}-1)}=x,\\& \lim_{y\to -\infty}\frac{xe^{\bb y}}{1+x(e^{\bb y}-1)}=0\end{split}\nonumber}
we transform the initial integral into
\eqn{\label{eq:integral1}\begin{split}
&\frac{x^{\gamma}}{(1-x)^{\gamma-1}}\int_x^1f'(z)\cdot\left(\mui p_1\left(\frac{1-z}{z}\right)^{\gamma}+\muo q_1\left(\frac{1-z}{z}\right)^{\gamma-1}\right)\d z\\
&-\frac{(1-x)^{\gamma+1}}{x^{\gamma}}\int_0^xf'(z)\cdot\left(\mui p_2\left(\frac{z}{1-z}\right)^{\gamma}+\muo q_2\left(\frac{z}{1-z}\right)^{\gamma+1}\right)\d z,
\end{split}}
where we denote $\gamma:=1/(\bb w)$.
To get rid of the first integral in \eqref{eq:integral1} we multiply both sides
of equation \eqref{eq:main1} by $(1-x)^{\gamma-1}/x^{\gamma}$ and we differentiate it with respect to $x$.
After reordering and multiplying obtained equation by $x^{2\gamma+1}/(2\gamma(1-x)^{2\gamma-1})$ we derive
\eqn{\label{eq:main2}\begin{split}
&f'(x)\cdot\left[-\frac{\lambda}{2}\frac{x^{\gamma}}{(1-x)^{\gamma}}+\frac{(\gamma-1)\muo-\gamma\mui}{2\gamma}\frac{x^{\gamma+1}}{(1-x)^{\gamma}}\right]\\
&+f''(x)\cdot\Bigg[\frac{\lambda}{2\gamma}\frac{x^{\gamma+1}}{(1-x)^{\gamma-1}}-\frac{\muo-\mui}{2\gamma}\frac{x^{\gamma+2}}{(1-x)^{\gamma-1}}\\
&\q\q-\sbb\frac{\gamma-2+3x}{2\gamma}\frac{x^{\gamma+2}}{(1-x)^{\gamma-1}}\Bigg]\\
&+f'''(x)\cdot\sbb\frac{x^{\gamma+3}}{(1-x)^{\gamma-2}}\frac{1}{2\gamma}-c\frac{\gamma-1}{2\gamma}\frac{x^{\gamma+1}}{(1-x)^{\gamma+1}}\\
&+\int_0^xf'(z)\left(\mui p_2\left(\frac{z}{1-z}\right)^{\gamma}+\muo q_2\left(\frac{z}{1-z}\right)^{\gamma+1}\right)\d z=0.
\end{split}}
Now we differentiate the last equation with respect to $x$ to get rid of the last integral.
After reordering and multiplying by $2\gamma(1-x)^{\gamma+2}/x^{\gamma-1}$ we get:
\eqn{\label{eq:main3}\begin{split}
&y(x)\Big[-(1-x)\lambda\gamma^2+x(1-x)\Big(\muo(x(2\gamma q_2-\gamma+1)+\gamma^2-1)\\
&\q\q\q\q\q\q\q+\mui(x(\gamma-2\gamma p_2)+2\gamma p_2-\gamma^2-\gamma)\Big)\Big]\\
&+y'(x)\Big[x(1-x)^2(\lambda-2\lambda x)+x^2(1-x)^2\Big(\muo(3x-3)+\mui(2-3x)\\
&\q\q\q\q\q\q\q\q\q-\sbb(\gamma^2-12x^2+15x-4)\Big)\Big]\\
&+y''(x)\left[x^2(1-x)^3\lambda-x^3(1-x)^3\left(\muo-\mui+(8x-5)\sbb)\right)\right]\\
&+y'''(x)x^4(1-x)^4\sbb-c(\gamma^2-1)x=0
\end{split}}
for
$$y(x):=f'(x).$$
By inspection one can show that the above nonhomogeneous equation has two singular points: $x=0$ and $x=1$. Both are regular but the latter is of the first kind while the former of second. The theory of such singular ordinary equations is well-developed and states that our equation has a unique solution which can be represented by the formal power series (see for ex. \cite{coddington1955theory}, Chapter 5)
$$
	y(x) = \sum_{n=1}^{\infty} c_n x^n. 	
$$
Further, classical results state that the above series is in general convergent to the actual solution but only in the asymptotic sense as $x\rightarrow 0^+$. On the other hand, the absolute convergence of the above can also be established but only in some particular cases (for the seminal papers see \cite{turrittin1955convergent,harris1969holomorphic}). A series of interesting alternative theorems has also been established \cite{grimm1975alternative}.

First, we will show that there exists a point $A^*$ such that the solution of (\ref{eq:main3}) satisfies $y(A^*) = -1$. To start, note that putting $x = 1$ in (\ref{eq:main3}) will yield a contradiction unless $y$ blows up according to
$$
	y(x) \sim -\frac{b}{1-x} \quad \text{as} \quad x\rightarrow 1^-,
$$
for some constant $b$ that can be found by plugging the above \textit{ansatz} into (\ref{eq:main3}). By a straightforward calculation it can be found as
$$
	b = \frac{2 c \left(\gamma ^2-1\right)}{\beta_0 ^2 \gamma ^2 \sigma ^2-\beta_0 ^2 \sigma ^2+2 \gamma ^2 \lambda +2 \gamma ^2 \mui -2 \gamma ^2 \muo+2 \gamma  \muo-2 \lambda -2 \mui+4 \muo-4 \gamma \muo q_2},
$$
which is positive by the assumption. We see that $y(x) \rightarrow -\infty$ as $x\rightarrow 1^-$. By continuity, there exists a point $A^*$ with the sought property. Due to the monotonicity of $y$ (concavity of $f$, Lemma \ref{lem:concave}) this $A^*$ is unique.

We have shown that $f=y'$ satisfies the smooth fit condition (\ref{smooth_fit}). The continuous fit (\ref{cont_fit}) can be established as follows. First, by integration we have
$$
	f(x) = A + \int_0^x y(z) \d z.
$$
Hence, in order to satisfy the continuous fit we must impose
$$
	A = 1 - A^* - \int_0^{A^*} y(z) \d z.
$$
The last step is to ascertain whether the constant $A$ is well-defined, i.e. $0\leq A \leq 1$. Of course, $A$ cannot be negative since then by the monotonicity of $f$ we would have $f(x) < 0 $ for all $x\in[0,1]$. Moreover, it cannot be greater than $1$ since by the assumption the line $1 - x$ is tangent to $f$ at $A^*$. Because $f$ is concave its graph must lie below every tangent. Hence $A \leq 1$.

We have thus proved that there exists a unique function which is a solution of (\ref{eq:main2}) and satisfies (\ref{cont_fit})-(\ref{norm_ent}). Hence, the optimal value function $V^*(x)$ can be calculated by the formula
$$
	V^*(x) = 1 - A^* - \int_x^{A^*} y(z) \d z,
$$
where $A^*$ is such that $y(A^*) = -1$.

Finally, we will simplify the form of the solution by the reduction of the singular point at $x = 1$. The main reason of the following transformation is to facilitate the numerical procedure by avoiding resolving the logarithmic blow up. To this end substitute
$$
	u(x) = (1-x) y(x).
$$
If we write (\ref{eq:main3}) compactly as
$$
	\sum_{k=0}^{3} (1-x)^{k+1} a_k(x) y^{(k)}(x) = c(\gamma^2-1)x,
$$
then it will be equivalent to
\eqn{
	\sum_{k=0}^{3} (1-x)^{k+1} b_k(x) u^{(k)}(x) = c(\gamma^2-1)x(1-x),
\label{eq:main3u}
}
where
$$
	b_k(x) = \frac{1}{k!}\sum_{i=k}^{3} i! \; a_i(x).
$$
Notice that when defining $p_k(x)$ we have explicitly factored the polynomial $(1-x)^{k+1}$. The above formulas can be verified by a direct calculation and the fact that
$$
	y^{(i)}(x) = \frac{\d}{\d x^i}	\left(\frac{u(x)}{1-x}\right) = \sum_{k=0}^i \binom{i}{k} \frac{(i-k)!}{(1-x)^{i-k+1}} u^{(k)}(x).
$$
We can see that both the left- and right-hand sides of (\ref{eq:main3u}) vanish for $x=1$ and hence $u(x)$ is finite and has a convergent Taylor expansion at $x=1$.

Now, in order to actually solve (\ref{eq:main3u}) we have to impose the initial conditions for $u(x)$. From the normal entrance condition (\ref{norm_ent}) we obviously have $u(0) = 0$. The values $u'(0)$ and $u''(0)$ can be found by substitution of $u(x) = \alpha_1 x + \alpha_2 x^2 + ...$ into (\ref{eq:main3u}) and comparing the terms with respective powers of $x$. By tedious algebra we can find that (\ref{eq:initial}) holds.
\exit

\section*{References}

\bibliography{detection_bib}

\end{document}